\documentclass[12pt]{article}
\usepackage[centertags]{amsmath}
\usepackage{amsfonts}
\usepackage{amssymb}
\usepackage{amsthm}
\usepackage{amsmath}
\usepackage{color}
\usepackage{url}
\usepackage[vcentermath,enableskew,noautoscale,stdtext]{youngtab}
\theoremstyle{plain}
\newtheorem{thm}{Theorem}
\newtheorem{cor}[thm]{Corollary}
\newtheorem{lem}[thm]{Lemma}
\newtheorem{prop}[thm]{Proposition}
\theoremstyle{definition}
\newtheorem{defn}{Definition}
\theoremstyle{definition}
\newtheorem{ex}{Example}
\newtheorem{rem}{Remark}

\def    \bs     {\boldsymbol}
\def    \C      {\mathbb{C}}
\def    \Z      {\mathbb{Z}}
\def    \N      {\mathbb{N}}

\title{Hall-Littlewood symmetric functions via Yamanouchi toppling game}
\author{Robert Cori\footnote{LaBRI Universit\'e Bordeaux 1, Email: \textsf{cori@labri.u-bordeaux.fr}}, Pasquale Petrullo\footnote{Universit\`a degli Studi della Basilicata, Email: \textsf{p.petrullo@gmail.com}} and Domenico Senato\footnote{Universit\`a degli Studi della Basilicata, Email: \textsf{domenico.senato@unibas.it}}}
\date{}
\begin{document}
\maketitle
\begin{abstract}
We define a solitary game, the Yamanouchi toppling game, on any connected graph of $n$ vertices. The game arises from the well-known chip-firing game when the usual relation of equivalence defined on the set of all configurations is replaced by a suitable partial order. The set all firing sequences of length $m$ that the player is allowed to perform  in the Yamanouchi toppling game  is shown to be in bijection with all standard Young tableaux whose shape is a partition of the integer $m$ with at most $n-1$ parts. The set of all configurations that a player can obtain from a starting configuration is encoded in a suitable formal power series. When the graph is the simple path and each monomial of the series is replaced by a suitable Schur polynomial, we prove that such a series reduces to Hall-Littlewod symmetric polynomials. The same series provides a combinatorial description of orthogonal polynomials when the monomials are replaced by products of moments suitably modified.
\end{abstract}
\vspace{0.5cm}
\section{Introduction}
In~\cite{BLS91} A. Bj\"orner, L. Lov\'asz and P. Schor have studied a solitary game called the chip-firing game which is closely related to the sandpile model  of Dhar~\cite{Dhar}. In more recent papers some developments around this game were proposed. Musiker~\cite{Mus} introduced an unexpected relationship with elliptic curves, Norine and Baker~\cite{BN07} by means of an analogous game proposed a Riemann-Roch formula for graphs, for which Cori and Le Borgne~\cite{CLB13} presented a purely combinatorial description. An  algebraic presentation of the theory can be found in \cite{Biggs,dharRuelleVerma,Perk}.

Given a graph $\mathcal{G}$ with vertices $v_1,v_2,\ldots,v_n$, one may consider any array $\alpha=(\alpha_1,\alpha_2,\ldots,\alpha_n)$ of integers as a configuration associating to each vertex $v_i$ the weight $\alpha_i$. Suitable moves, here denoted by
$T_1,T_2,\ldots,T_n$ and called topplings, can be performed in the
game in order to change the starting configuration $\alpha$ into a
new configuration $\beta$. Such moves can be reversed and this
defines a relation of equivalence on $\mathbb{Z}^n$, here called toppling equivalence. The combinatorial interest of such a relation is grounded on its connections with several well-known combinatorial objects such as parking functions and Dick paths. In this paper we investigate more on this combinatorial game, which we refer to as the toppling game, by disclosing a wide range of connections with classical orthogonal polynomials and symmetric functions that we have outlined in \cite{CPS}.

Let $\alpha,\beta\in\Z^n$ and assume that $\beta$ is obtained from $\alpha$ by successively performing topplings
$T_{i_1},T_{i_2},\ldots,T_{i_l}$. Then, we say that $(T_{i_1},T_{i_2},\ldots,T_{i_l})$ is an $\alpha,
\beta$-\textit{toppling sequence} and denote by $\mathcal{T}_{\alpha,\beta}$ the set of all such sequences. Given $(T_{i_1},T_{i_2},\ldots,T_{i_l})\in\mathcal{T}_{\alpha,\beta}$ then it is easily seen that $(T_{i_{w(1)}},T_{i_{w(2)}} \ldots,T_{i_{w(l)}})\in\mathcal{T}_{\alpha,\beta}$ for any permutation $w$ of $1,2\ldots,l$. Moreover, if $1\leq k\leq l$ and if $\alpha^{(k)}$ is the configuration obtained from $\alpha$ via
$(T_{i_1},T_{i_2},\ldots,T_{i_k})$, then it is plain that $\alpha\equiv\alpha^{(k)}$, with $\equiv$ denoting toppling equivalence. We will focus our attention on a restricted class $\mathcal{Y}_{\alpha,\beta}\subseteq\mathcal{T}_{\alpha,\beta}$ of toppling sequences that arise when toppling equivalence $\equiv$ is replaced with a new relation $\leq$ defined on $\Z^n$. A first crucial fact is that $\leq$ is a partial order. Thus, instead of the whole classes of equivalent configurations, one may consider order ideals $\mathcal{H}_\alpha$'s generated by all $\alpha$'s. More concretely, one may also thinks $\mathcal{H}_\alpha$ as the set of all configurations $\beta$'s such that $\mathcal{Y}_{\alpha,\beta}\neq \varnothing$. In particular, one has $(T_{i_1},T_{i_2} \ldots,T_{i_l})\in\mathcal{Y}_{\alpha,\beta}$ if and only the configuration $\alpha^{(k)}$, obtained from $\alpha$ via $(T_{i_1},T_{i_2},\ldots,T_{i_k})$, satisfies $\alpha^{(k)}\leq \alpha$ for all $1\leq k\leq l$. Therefore, an explicit characterization of $\mathcal{Y}_{\alpha,\beta}$ states that $(T_{i_1},T_{i_2},\ldots,T_{i_l})\in\mathcal{Y}_{\alpha,\beta}$ if and only if $i_1i_2\ldots i_l$ is a suitable Yamanouchi word over the alphabet of positive integers. This is why any  sequence in $\mathcal{Y}_{\alpha,\beta}$ will be called a \textit{Yamanouchi toppling  sequence}.

Topplings also acts on the set $\{x^\alpha\}_{\alpha\in \Z^n}$, of monomials of the type $x^\alpha=x_1^{\alpha_1}x_2^{\alpha_2}\cdots x_n^{\alpha_n}$, via $T_i\cdot x^\alpha=x^{T_i(\alpha)}$. In turn, this induces an action of the toppling group $G$ (i.e. the group generated by $T_1,T_2,\ldots,T_n$) on the ring of formal series $\Z[[x_1^{\pm 1},x_2^{\pm 1}, \ldots, x_n^{\pm 1}]]$. From this perspective, we may identify $\mathcal{H}_\alpha$ with the series
\[\mathcal{H}_\alpha(x)=\sum_{\beta\leq\alpha}x^\beta.\]
Since $\leq$ is a partial order, then $\{\mathcal{H}_\alpha(x)\}_{\alpha\in\Z^n}$ is a basis of $\Z[[x_1^{\pm 1},x_2^{\pm 1}, \ldots, x_n^{\pm 1}]]$. By setting $T_{[i]}=T_1T_2\cdots T_i$ we can prove that the operator $\tau$, defined by
\[\tau=\prod_{i=1}^{n-1}\frac{1}{1-T_{[i]}},\]
satisfies
\[\tau\cdot x^\alpha=\mathcal{H}_\alpha(x),\]
for all $\alpha\in\Z^n$. A deformed version of $\tau$, denoted $\hat{\tau}$, arises when elements of type $T_{[i,j]}=T_{[i]}T_{[i+1]}\cdots T_{[j-1]}$ are taken into consideration. More precisely, we set
\[\hat{\tau}=\prod_{1\leq i< j\leq n}\frac{1}{1-T_{[i,j]}},\]
and obtain a further basis $\{\hat{\mathcal{H}}_\alpha(x)\}_{\alpha\in\Z^n}$ satisfying
\[\hat{\mathcal{H}}_\alpha(x)=\hat{\tau}\cdot x^\alpha=\sum_{\beta\leq\alpha}C_{\alpha,\beta}x^\beta,\]
with $C_{\alpha,\beta}$ counting the number of pairwise distinct
decompositions in terms of the generators $T_{[i,j]}$'s of the unique
$g\in G$ such that $g(\alpha)=\beta$. At this point, one may introduce parameters $z_1,z_2,z_3,q$ in order to keep track of the joint distribution of certain statistics
$\ell_1,\ell_2,\ell_3,d$ defined on the set of all decompositions
of any element in the toppling group. This is possible via a further deformation $\hat{\tau}(z_1,z_2,z_3,q)$ of $\hat{\tau}$, which leads to a parametrized version $\hat{\mathcal{H}}_\alpha(z_1,z_2,z_3,q;x)$ of $\hat{\mathcal{H}}_\alpha(x)$. More precisely, we set
\[\hat{\tau}(z_1,z_2,z_3,q)=\prod_{1\leq i<j\leq n}\frac{1-(1-q)T_{[i,j]}z_3z_2^{j-i}z_1^{\binom{j}{2}-\binom{i}{2}}}{1-T_{[i,j]}z_3z_2^{j-i}z_1^{\binom{j}{2}-\binom{i}{2}}},\]
and obtain
\[\hat{\mathcal{H}}_\alpha(z_1,z_2,z_3,q;x)=\hat{\tau}(z_1,z_2,z_3,q)\cdot x^\alpha=\sum_{\beta\leq\alpha}C_{\alpha,\beta}(z_1,z_2,z_3,q;x)x^\beta.\]
Since $\hat{\tau}(z_1,z_2,z_3,q)^{-1}=\hat{\tau}\left(z_1,z_2,(1-q)z_3,\frac{q}{q-1}\right)$,
then an explicit description of the series
\[\hat{\mathcal{K}}_\alpha(z_1,z_2,z_3,q;x)=\hat{\tau}(z_1,z_2,z_3,q)^{-1}\cdot x^\alpha\]
is obtained for free,
\[\hat{\mathcal{K}}_\alpha(z_1,z_2,z_3,q;x)=\sum_{\beta\leq \alpha}
\left(\sum_{\ell_1,\ell_2,\ell_3,d} z_1^{\ell_1}z_2^{\ell_2}z_3^{\ell_3}(1-q)^{\ell_3-d}(-q)^d\right) x^\beta.\]
In the summation above, the values of $\ell_1,\ell_2,\ell_3,d$ range over all pairwise distinct factorizations $g=T_{[i_1,j_1]}T_{[i_2,j_2]}\cdots$ of the unique $g\in G$ such that $g(\alpha)=\beta$. Since $\ell_1\geq \ell_2\geq\ell_3\geq d\geq 0$ then the
coefficient of $x^\beta$ in $\hat{\mathcal{K}}_\alpha(z_1,z_2,z_3,q;x)$, as well as that in
$\hat{\mathcal{H}}_\alpha(z_1,z_2,z_3,q;x)$, is a polynomial with
integer coefficients.

Noteworthy applications of this theory arise when Yamanouchi toppling is performed on the simple path with edges $\{v_1,v_2\},\{v_2,v_3\},\ldots,\{v_{n-1},v_{n}\}$. On one hand, the operator $\hat{\tau}$ reduces to a certain lowering operator arising within theory of symmetric functions and mapping Schur functions into Hall-Littlewood symmetric functions \cite{G92,M95}. As a consequence, the series $\hat{\mathcal{H}}_\alpha(z_1,z_2,z_3,q;x)$ and $\hat{\mathcal{K}}_\alpha(z_1,z_2,z_3,q;x)$ reduce to Hall-Littlewood symmetric functions when $\alpha$ is a partition, when $q,z_1,z_2,z_3$ are suitably specialized, and when each $x^\beta$ is suitably replaced by a Schur function. In turn, this enables us to define an analogue of Hall-Littlewood symmetric functions for any connected graph, thus opening the way to a systematic study of the matter. On the other hand, we can also prove that both $\hat{\mathcal{H}}_\alpha(q,z_1,z_2,z_3;x)$ and $\hat{\mathcal{K}}_\alpha(q,z_1,z_2,z_3;x)$ reduces to the $(n-1)$th orthogonal polynomial $p_{n-1}(t)$, associated with a given linear functional with moments $a_i$'s, whenever $\alpha=(n-1,n-1,\ldots,n-1,0)$ and each $x^\beta$ is replaced by $a_{\beta_1}a_{\beta_2}\cdots a_{\beta_{n-1}}t^{\beta_{n}}$. As an example, Hermite polynomials, Poisson-Charlier polynomials, Jacobi polynomials and any other classical orthogonal basis of the ring of polynomials in a single variable can be obtained by choosing the right sequence of moment (i.e. the right linear functional). Again, an analogue of classical orthogonal polynomials can be defined for any connected graph, with the chip-firing game concurring in giving a new combinatorial ground in common with symmetric functions.
\section{Configurations on graphs, toppling game and Yamanouchi words}
Here and in the following, by a graph $\mathcal{G}$ we will always mean a connected graph $\mathcal{G}=(V,E)$, with set of vertices $V=\{v_1,v_2,\ldots,v_n\}$, and with at most one edge $\{v_i,v_j\}$ for all $1\leq i< j\leq n$. If the edge $\{v_i,v_j\}$ belongs to $E$ then $v_i$ and $v_j$ will be said neighbors. A \textit{configuration} on $\mathcal{G}$ is a map, $\alpha\colon v_i\in V\mapsto \alpha(v_i)\in\mathbb{Z}$, associating to each vertex $v_i$ an
integral weight $\alpha(v_i)$. If we set $\alpha_i=\alpha(v_i)$
then we may identify any configuration $\alpha$ with the array
$(\alpha_1,\alpha_2,\ldots,\alpha_n)$. Henceforth, if $1\leq i\leq n$, then $\epsilon_i=(\delta_{i1},\delta_{i2},\ldots,\delta_{in})$ will denote the configuration associating $v_i$ with $1$ and
$v_j$ with $0$ if $j\neq i$. A \textit{toppling} of the vertex
$v_i$ is a map
$T_i\colon\mathbb{Z}^n\to\mathbb{Z}^n$ defined by
\begin{equation}\label{T}
T_i(\alpha)=\alpha+\Delta_i,
\end{equation}
where
\[\Delta_i=\left(\sum_{\{v_j,v_i\}\in E}\epsilon_j\right)-d_i\epsilon_i,\]
and $d_i=\big|\{v_j\,|\,\{v_i,v_j\}\in E\}\big|$ is the \textit{degree} of $v_i$. Roughly speaking, the map $T_i$ increases by $1$ the weight $\alpha_j$ of each neighbor of $v_i$, and simultaneously decreases by $d_i$ the weight $\alpha_i$. As a consequence, the \textit{size} $|\alpha|=\alpha_1+\alpha_2+\cdots+\alpha_n$ of any $\alpha\in\Z^n$ is preserved by each toppling $T_i$.

One may look at each $T_i$ as a move of a suitable combinatorial game on the graph $\mathcal{G}$, that will be referred to as the toppling game. More precisely, assume that a starting configuration $\alpha=(\alpha_1,\alpha_2,\ldots,\alpha_n)$ is given on $\mathcal{G}$, then label each vertex $v_i$ with its own weight $\alpha_i$. By \lq\lq{firing}\rq\rq the vertex $v_i$ the starting configuration $\alpha$ is changed into a new configuration $\beta=T_i(\alpha)$. A \textit{toppling  sequence} on the graph $\mathcal{G}$ simply is a finite sequence of fired vertices $(v_{i_1},v_{i_2},\ldots,v_{i_l})$ or, equivalently, a finite sequence $(T_{i_1},T_{i_2},\ldots,T_{i_l})$ of moves. We say that $(T_{i_1},T_{i_2},\ldots,T_{i_l})$ is an $\alpha,\beta$-\textit{toppling sequence} to express that $\alpha$ can be changed into $\beta$ by successively performing the corresponding moves, for short $\beta=T_{i_l}T_{i_{l-1}}\cdots T_{i_1}(\alpha)$. It can be shown that a $\alpha,\beta$-toppling  sequence exists if and only if a $\beta,\alpha$-toppling  sequence exists. Then a relation of equivalence, called \textit{toppling equivalence}, can be defined on $\Z^n$ by setting $\alpha\equiv\beta$ if and only if an $\alpha,\beta$-toppling  sequence exists. Note that, the player of an $\alpha,\beta$-toppling  sequence $(T_{i_1},T_{i_2},\ldots,T_{i_l})$ passes through intermediate configurations $\alpha=\alpha^{(0)},\alpha^{(1)},\alpha^{(2)},\ldots,\alpha^{(l)}=\beta$ defined by $\alpha^{(k)}=T_{i_k}T_{i_k-1}\cdots T_{i_1}(\alpha)$ and satisfying $\alpha\equiv\alpha^{(k)}$. We are going to define an analogous game by replacing the toppling equivalence with a different relation on $\mathbb{Z}^n$.

From \eqref{T} one easily recover $T_iT_j=T_jT_i$ for all $1\leq i,j\leq n$. This means that the set $\mathcal{T}_{\alpha,\beta}$ of all $\alpha,\beta$-toppling  sequences is closed under permutation of the topplings involved. In particular, this means that $\alpha\equiv\beta$ if and only  if there exists $a\in \Z^n$ such that $T^a(\alpha)=\beta$, with $T^{a}=T_1^{a_1}T_2^{a_2}\cdots T_{n}^{a_n}$ and $T_i^{a_i}(\alpha)=\alpha+a_i\Delta_i$.
\begin{defn}[Toppling dominance]\label{Def:topp_dom}
Let $\mathcal{G}$ be a graph and let $\alpha,\beta\in\Z^n$. We say that $\alpha$ \textit{dominates} $\beta$ \textit{with respect to} $\mathcal{G}$, written $\beta\leq\alpha$, if and only if
\[\beta=T^\lambda(\alpha) \text{ with }\lambda\in\N^n\text{ and }\lambda_1\geq \lambda_2\geq\ldots\geq\lambda_n.\]
\end{defn}
Now, assume $\beta\leq \alpha$ and assume that a player is asked to perform, if possible, an $\alpha,\beta$-toppling  sequence $(T_{i_1},T_{i_2},\ldots,T_{i_l})$ which obeys the following prescription: for all $1\leq k\leq l$, if $\alpha^{(k)}=T_{i_k}T_{i_{k-1}}\cdots T_{i_1}(\alpha)$ then $\alpha^{(k)}\leq \alpha$. Henceforth, we will denote by $\mathcal{Y}_{\alpha,\beta}$ the set of all toppling sequences in $
\mathcal{T}_{\alpha,\beta}$ that obey such a prescription.
\begin{ex}
Let $\mathcal{G}$ be the complete graph with vertices $v_1,v_2,v_3,v_4,v_5$, then let $\alpha=(5,-3,0,1,-4)$ and $\beta=(-6,-4,4,5,0)$. Straightforward computations show that both
\[(T_1,T_1,T_1,T_1,T_2,T_2,T_3,T_4,T_5)\text{ and }(T_5,T_4,T_3,T_2,T_2,T_1,T_1,T_1,T_1)\]
are $\alpha,\beta$-toppling  sequences. Nevertheless, the former  sequence is in $\mathcal{Y}_{\alpha,\beta}$ instead of the latter which is in $\mathcal{T}_{\alpha,\beta}\setminus\mathcal{Y}_{\alpha,\beta}$. Moreover, the former  sequence is not of minimal length since we also have $(T_1,T_1,T_1,T_2)\in\mathcal{Y}_{\alpha,\beta}$.
\end{ex}
Hence, a first problem the player is going to face off is that of characterizing the set $\mathcal{Y}_{\alpha,\beta}$. A second matter is that of determining those sequences in $\mathcal{Y}_{\alpha,\beta}$ involving the minimum
number of moves. One may easily realize that the $\alpha,\beta$-toppling  sequence $(T_{i_1},T_{i_2},\ldots,T_{i_l})$ is in $\mathcal{Y}_{\alpha,\beta}$ if and only the following condition is satisfied: for all $1\leq k\leq n$ and for all $2\leq i\leq n$, the number of occurrences of $T_i$ in $(T_{i_1},T_{i_2},\ldots,T_{i_k})$ does not exceed the number of occurrences of $T_{i-1}$. So, if $\alpha$ is fixed and if we identify $(T_{i_1},T_{i_2},\ldots,T_{i_l})$ with the word $i_1i_2\ldots i_l$, then $\bigcup_{\beta}\mathcal{Y}_{\alpha,\beta}$ exactly corresponds to the set of all Yamanouchi words over $\{1,2,\ldots,n\}$. Recall that, associated with each Yamanouchi word $w=i_1\,i_2\,\ldots\,i_l$, and hence with each Yamanouchi toppling  sequence $(T_{i_1},T_{i_2},\ldots,T_{i_l})$, there is an integer partition $\lambda(w)=(\lambda_1,\lambda_2,\ldots)$ whose $i$th part $\lambda_i$ equals the number of
occurrences of $i$ in $w$. A suitable filling of the Young diagram
of $\lambda(w)$ yields a coding of $w$ in terms of a standard
Young tableau. More precisely, the tableau associated with
$w$ is the unique tableau of shape $\lambda(w)$ whose $i$th row
stores all $j$'s such that $i_j=i$. This provides a bijection between the set of all Yamanouchi words of $l$ letters and the set of all standard Young tableaux of $l$ boxes \cite{St}. For instance, for the Yamanouchi word $w=1\,1\,2\,1\,3\,2\,4$ we recover a standard Young tableau of shape $\lambda(w)=(3,2,1,1)$,
\[w\mapsto\young(124,36,57).\]
Note that, if the Yamanouchi words $w$ and $w'$ agree up to the order then $\lambda(w)=\lambda(w')$, so that the corresponding toppling  sequences end at the same configuration $\beta$ whenever their starting configuration is the same. However, the converse is not true. In fact, we have already noticed that, if $\mathcal{G}$ is the complete graph with five vertices, the words $w=1\,1\,1\,1\,2\,2\,3\,4\,5$ and $w'=1\,1\,1\,2,$ both change $\alpha=(5,-3,0,1,4)$ into $\beta=(-6,-4,4,5,0)$. However, their corresponding Young tableaux are of different size,
\[w\mapsto\young(1234,56,7,8,9) \text{ and }w'\mapsto\young(123,4).\]
So, in order to get an explicit characterization of each set $\mathcal{Y}_{\alpha,\beta}$ we need a slightly deeper investigation.
\section{The toppling group}
Assume a graph $\mathcal{G}$ is given and denote by $T_1,T_2,\ldots,T_n$ the corresponding toppling maps. The \textit{toppling group} associated with $\mathcal{G}$ is the group $G$ generated by $T_1,T_2,\ldots,T_n$. Since $T_iT_j=T_jT_i$ for all $1\leq i<j\leq n$ then $G$ is commutative. In particular, this says that all $g\in G$ may be expressed in terms of the $T_i$'s as $T^a=T_1^{a_1}T_2^{a_2}\cdots T_{n}^{a_n}$, for a suitable array of integers $a=(a_1,a_2,\ldots,a_n)\in \mathbb{Z}^n$. On the other hand, it is also easy to check that $T_1T_2\cdots T_n(\alpha)=\alpha$ for all $\alpha\in\mathbb{Z}^n$ and for all $\mathcal{G}$. Thus, we have $T_1T_2\cdots T_n=1$, with $1$ denoting the identity of $G$. As a consequence, if $a-b=k(\epsilon_1+\epsilon_2+\cdots+\epsilon_n)$ for some $k\in\mathbb{Z}$, then $T^a=T^b$. This means that each $g\in G$ admits a presentation $T^a$ with $a\in\N^n$. In fact, if $g=T^b$ and if $b\notin\N^n$, then we set $k=\min\{b_1,b_2,\ldots,b_n\}$ and choose $a=b-k(\epsilon_1+\epsilon_2+\cdots+\epsilon_n)$. Now, it is $a\in\N^n$ and $T^a=T^b(T_1T_2\cdots T_n)^{-k}=g$. For instance, if $b=(-3,-1,0,2,0,0,4,0)$ then $k=-3$, $a=(0,2,3,5,3,3,7,3)$ and finally
\[T^b=T^a=T_2^2T_3^3T_4^5T_5^3T_6^3T_7^7T_8^3.\]
In order to show that $T^a=T^b$ if and only if $a-b=k(\epsilon_1+\epsilon_2+\cdots+\epsilon_n)$ with $k\in\mathbb{Z}^n$ we need a preliminary lemma.
\begin{lem}\label{lem}
We have
\[T^a=1 \text{ if and only if }a=k(\epsilon_1+\epsilon_2+\cdots+\epsilon_n).\]
\end{lem}
\begin{proof}
If $a=k(\epsilon_1+\epsilon_2+\cdots+\epsilon_n)$ then it is obvious that $T^a=1$. Conversely, assume $T^a=1$ and $a\in \N^n$. Let $k=\min\{a_1,a_2,\ldots,a_n\}$ and set $b=a-k(\epsilon_1+\epsilon_2+\cdots+\epsilon_n)$. We obtain $T^a=T^b$ and $\min\{b_1,b_2,\ldots,b_n\}=0$. Assume $b_i=0$ and consider $T^b$ as a toppling sequence. Since $b_i=0$, then $v_i$ is not fired. Moreover, since the toppling sequence does not change the weight of $v_i$, then none of the neighbors of $v_i$ have been fired. This means $b_j=0$ whenever $v_j$ is a neighbor of $v_i$. On the other hand, we may repeat the same reasoning for each neighbor of $v_i$ and, since the graph is connected, in a finite number of steps we will have $b_i=0$ for all $1\leq i\leq n$. This exactly means $a=k(\epsilon_1+\epsilon_2+\cdots+\epsilon_n)$.
\end{proof}
\begin{rem}\label{unique}
Let $a,b \in\Z^n$ and assume  $T^a(\alpha)=T^b(\alpha)$ for some $\alpha$, equivalently, $T^{a-b}(\alpha)=\alpha$. If $c=(a-b)-k(\epsilon_1+\epsilon_2+\cdots+\epsilon_n)$ and if $k=\min\{a_1-b_1,a_2-b_2,\ldots,a_n-b_n\}$, then $T^{a-b}=T^c$ and $\min\{c_1,c_2,\ldots,c_n\}=0$. Now, we may carry out a same reasoning as in the proof of the lemma above obtaining $T^c=1$ and then $T^a=T^b$. That is $T^a(\alpha)=T^b(\alpha)$ if and only if $T^a=T^b$ and this means that if $\alpha\equiv\beta$ then there is a unique $g\in G$ such that $g(\alpha)=\beta$.
\end{rem}
A first consequence of Lemma~\ref{lem} is that the only relations satisfied by the generators of $G$ are $T_iT_j=T_jT_i$ and $T_1T_2\cdots T_n=1$. This means that the group algebra $\mathbb{C}[G]$ of $G$ is isomorphic to the ring
\[\frac{\mathbb{C}[x_1,x_2,\ldots,x_n]}{\langle 1-x_1x_2\cdots x_n\rangle},\]
of polynomials $\mathbb{C}[x_1,x_2,\ldots,x_n]$ modulo the ideal generated by $1-x_1x_2\cdots x_n$.
Moreover this provides an explicit characterization of all distinct presentations of any element in the toppling group $G$ in terms of the generators $T_1,T_2,\ldots,T_n$.
\begin{thm}\label{thm}
For all $a,b\in\N^n$ we have
\[T^a=T^b \text{ if and only if }b-a=k(\epsilon_1+\epsilon_2+\cdots+\epsilon_n),\]
for a suitable $k\in\Z$.
\end{thm}
\begin{proof}
Let $h=\max\{a_1,a_2,\ldots,a_n\}$ and set
$\tilde{a}=h(\epsilon_1+\epsilon_2+\cdots+\epsilon_n)-a$. Clearly
$T^aT^{\tilde{a}}=1$ and, being $T^a=T^b$, also
$T^bT^{\tilde{a}}=1$. By virtue of Lemma~\ref{lem} we deduce
$b+\tilde{a}=j(\epsilon_1+\epsilon_2+\cdots+\epsilon_n)$ and then
$b=a+k(\epsilon_1+\epsilon_2+\cdots+\epsilon_n)$ with $k=j-h$.
\end{proof}
Now, once $g\in G$ and $a\in\N^n$ are chosen such that $g=T^a$, we may set $k=\min\{a_1,a_2,\ldots,a_n\}$ and define
$b=a-k(\epsilon_1+\epsilon_2+\cdots+\epsilon_n)$. Clearly
$T^a=T^b$ and $b$ is the unique element in $\N^n$ of minimal size
with this property. In other words, $T^b$ is the unique
\textit{reduced decomposition} of $g=T^a$. Henceforth, we will denote by $I_n$ the set of all $a\in\mathbb{N}^n$ satisfying $\min\{a_1,a_2,\ldots,a_n\}=0$. Moreover, we denote by $P_n$ the set of all $\lambda=(\lambda_1,\lambda_2,\ldots,\lambda_n)\in I_n$ satisfying $\lambda_1\geq\lambda_2\geq \ldots\geq\lambda_n$, hence $\lambda_n=0$. Note that the map $a\in I_n\mapsto T^a\in G$ is a bijection. Furthermore, if zero entries are ignored then $P_n$ can be identified with the set of all integer partitions with at most $n-1$ parts. Finally, we can characterize each $\mathcal{Y}_{\alpha,\beta}$ in an explicit way.
\begin{thm}\label{thm:Yam_main}
If $\alpha,\beta\in\mathbb{Z}^n$ are such that $\beta\leq \alpha$ then there exists a unique $\lambda\in P_n$ such that
\[T^\lambda(\alpha)=\beta,\]
and all sequences in $\mathcal{Y}_{\alpha,\beta}$ of minimal length are those associated with standard Young tableaux of shape $\lambda$
\end{thm}
\begin{proof}
If $\beta\leq \alpha$ then there exists a unique $T^\eta\in G$ with $\eta_1\geq\eta_2\geq \ldots\geq\eta_n\geq 0$ and $\beta=T^\eta(\alpha)$. The reduced decomposition of $T^\eta$ is given by $T^\lambda$, where $\lambda=\eta-\eta_n(\epsilon_1+\epsilon_2+\cdots+\epsilon_n)$. Hence, $(T_{i_1},T_{i_2},\ldots,T_{i_l})\in\mathcal{Y}_{\alpha,\beta}$ is an $\alpha,\beta$-Yamanouchi toppling sequence of minimal length if and only if $i_1i_2\ldots i_l$ is Yamanouchi of type $\lambda$, that is if and only if it associated with some standard Young tableaux of shape $\lambda$.
\end{proof}
\begin{cor}
Let $\alpha,\beta\in\Z^n$ and assume $\beta=T^\lambda(\alpha)$ for some $\lambda\in P_n$. Then, all Yamanouchi toppling sequences associated with standard Young tableaux of shape $\mu$ are in $\mathcal{Y}_{\alpha,\beta}$ if and only if $\mu=\lambda+k(\epsilon_1+\epsilon_2+\cdots+\epsilon_n)$, with $k\in\N$.
\end{cor}
\begin{proof}
Let $i_1i_2\cdots i_l$ be Yamanouhi of type $\mu$. We have $(T_{i_1},T_{i_2},\ldots,T_{i_l})\in\mathcal{Y}_{\alpha,\beta}$ if and only if $\beta=T^\mu(\alpha)=T^\lambda(\alpha)$. Then $T^\mu=T^\lambda$ and Theorem \ref{thm} assures us $\mu=\lambda+k(\epsilon_1+\epsilon_2+\cdots+\epsilon_n)$ for some $k\in\N$.
\end{proof}
\begin{ex}
Let $\mathcal{G}$ denote the complete graph with five vertices, then assign $\alpha=(5,-3,0,1,-4)$ and $\beta=(-6,-4,4,5,0)$. Since we have $\beta=T^\lambda(\alpha)$ for $\lambda=(3,1,0,0,0)$, then the minimum number of moves to pass from $\alpha$ to $\beta$ is $4=3+1$. All $\alpha,\beta$-Yamanouchi toppling sequences of minimal length are
\[(T_1,T_1,T_1,T_2),(T_1,T_1,T_2,T_1),(T_1,T_2,T_1,T_1).\]
They corresponds to the following standard Young tableaux,
\[\young(123,4)\quad\young(124,3)\quad\young(134,2)\]
\end{ex}
In the next section we will focus our attention on the set $\mathcal{H}_\alpha$ of all configurations that can be obtained from a given configuration $\alpha$ by means of any Yamanouchi toppling sequence.
\begin{rem}[On the weight lattice of type $A$]
In recent years, a general and beautiful algebraic theory of orthogonal polynomials have been developed in the framework of Hecke algebras associated with root systems \cite{M03}. For root systems of type $A$ the associated orthogonal polynomials are the well-known Macdonald symmetric polynomials \cite{M95}. By comparing with Kirillov \cite{K}, one may check that for the Weyl group $W=A_{n-1}$ (i.e. the symmetric group $\mathfrak{S}_n$) the set $I_n$ defined above can be identified with the weight lattice $P$. In turn, the set $P_n$ agrees with the set $P^+$ of dominant weights. The subalgebra $\C[P]^{W}$ of the group algebra $\C[P]$ turns out to be isomorphic to the quotient
\[\frac{\mathbb{C}[x_1,x_2,\ldots,x_n]}{\langle 1-x_1x_2\cdots x_n\rangle}.\]
The generators of $\C[P]$ are usually denoted as formal exponentials $e^\lambda$'s, with $\lambda\in P$. This suggests the identification $T_i=e^{\epsilon_i}$. As we will show in the following sections, the toppling game provides an alternative and purely combinatorial way to recover common ground for symmetric and orthogonal polynomials. However, it remains the  interesting question of a deeper understanding of possible connections between the toppling game and the whole theory developed in \cite{M03}.
\end{rem}
\section{Generating series of configurations}
For all $\alpha\in\Z^n$ we set
\[\mathcal{H}_\alpha=\{T^\lambda(\alpha)\,|\,\lambda\in P_n\}=\{\beta\,|\,\beta\leq \alpha\},\]
so that $\mathcal{H}_\alpha$ consists of all configurations that can be obtained from $\alpha$ by performing a Yamanouchi toppling sequence. Since the inverse of an element $T^\lambda$, with $\lambda\in P_n$, cannot be written in general as $T^\mu$ with $\mu\in P_n$, then toppling dominance is not a relation of equivalence.
\begin{prop}\label{P:order}
Toppling dominance is a partial order on $\mathbb{Z}^n$.
\end{prop}
\begin{proof}
It is plain that $\leq$ is a reflexive and transitive relation. Assume $\alpha\leq \beta$ and $\beta\leq \alpha$, so that $\beta=T^\lambda(\alpha)$ and $\alpha=T^\mu(\beta)$ with $\lambda,\mu\in P_n$. We deduce $\alpha=T^{\mu+\lambda}(\alpha)$ and so, via Lemma \ref{lem}, $\lambda+\mu=k(\epsilon_1+\epsilon_2+\cdots+\epsilon_n)$. By taking into account $\lambda_1\geq\lambda_2\geq \ldots\geq \lambda_n= 0$ and $\mu_1\geq\mu_2\geq\ldots \geq\mu_n=0$, we may write $\lambda_i=k-\mu_i\geq k-\mu_{i+1}=\lambda_{i+1}\leq \lambda_{i}$, which forces $\lambda_1=\lambda_2=\ldots=\lambda_n=0$ and also $\mu_1=\mu_2=\ldots=\mu_n=0$. We deduce $\alpha=\beta$, then $\leq$ is antisymmetric.
\end{proof}
In view of the proposition above, any set $\mathcal{H}_{\alpha}$ can be described as the principal order ideal generated by $\alpha$.

Set $\Z[[G]]=\Z[[T_1,T_2,\ldots,T_n]]$ and consider the following action of $\Z[[G]]$ on the ring $\Z[[x_1^{\pm 1},x_2^{\pm 1},\ldots,x_2^{\pm 1}]]$ of all formal series in $x_1^{\pm 1},x_2^{\pm 1},\ldots,x_2^{\pm 1}$. For all $a\in I_n$ and for all $\alpha\in\Z^n$ set $x^\alpha=x_1^{\alpha_1}x_2^{\alpha_2}\cdots x_n^{\alpha_n}$ and let
\[T^a\cdot x^\alpha=x^{T^a(\alpha)}.\]
By linear extension we obtain
\[\left(\sum_{a\in I_n}c_a T^a\right)\cdot\left(\sum_{\alpha\in\Z^n}d_\alpha x^\alpha\right)=\sum_{a\in\N^n}\sum_{\alpha\in\Z^n}c_a d_\alpha x^{T^a(\alpha)}.\]
Hence, any ideal $\mathcal{H}_\alpha$ uniquely determines the formal series
\[\mathcal{H}_\alpha(x)=\sum_{\beta\leq \alpha}x^\beta.\]
Consider the following element in $\Z[[G]]$,
\[\tau=\sum_{\lambda\in P_n}T^\lambda.\]
Since for all $\beta\in\mathcal{H}_\alpha$ there exists a unique $\lambda\in P_n$ such that $\beta=T^\lambda(\alpha)$, then we immediately recover
\[\tau\cdot x^\alpha=\sum_{\lambda\in P_n}T^\lambda\cdot x^\alpha=\sum_{\beta\leq\alpha}x^\beta=\mathcal{H}_\alpha(x).\]
Hence $\tau$ generates the whole order ideal $\mathcal{H}_\alpha$ by starting from the configuration $\alpha$. Note that $\tau$ does not depend on $\alpha$. Now, consider the following element in the toppling group,
\[T_{[i]}=T_1T_2\ldots T_i \text{ for all }i=1,2,\ldots,n-1.\]
Note that $T_{[i]}$ may be thought of as a Yamanouchi toppling sequence associated with a $1$-column standard Young tableau.
\begin{thm}\label{Tlambda}
Let $\lambda\in P_n$ and denote by $\lambda'=(\lambda'_1,\lambda'_2,\ldots)$ the conjugate of $\lambda$. Then, $T^\lambda$ has a unique
expression in terms of the elements $T_{[i]}$, more precisely:

\[T^\lambda=T_{[\lambda_1']}T_{[\lambda_2']}\cdots.\]
\end{thm}
\begin{proof}
Let $\ell(\lambda)$ and $\ell(\lambda')$ denote the lengths of $\lambda$ and $\lambda'$, respectively. Then we have

\[T^\lambda=\prod_{1\leq i\leq \ell(\lambda)}T_i^{\lambda_i}=\prod_{1\leq i\leq \ell(\lambda')}\left(\prod_{1\leq j\leq \lambda'_i}T_j\right)=\prod_{1\leq i\leq \ell(\lambda')}T_{[\lambda'_j]}.\]
\end{proof}
\begin{cor}
We have
\[\tau=\prod_{i=1}^{n-1}\frac{1}{1-T_{[i]}}.\]
\end{cor}
\begin{proof}
Consider the set $P'_n=\{\lambda'\,|\,\lambda\in P_n\}$. Clearly, $P'_n$ is nothing but the set of all integer partitions whose largest part does not exceed $n-1$ and the map $\lambda\in P_n\mapsto\lambda'\in P'_n$ is a bijection. Then, via Theorem \ref{Tlambda} we recover
\[\prod_{i=1}^{n-1}\frac{1}{1-T_{[i]}}=\sum_{\mu\in P'_n}T_{[\mu_1]}T_{[\mu_2]}\cdots =\sum_{\lambda\in P_n}T^\lambda=\tau.\]
\end{proof}
Now, we may write
\[\mathcal{H}_\alpha(x)=\prod_{i=1}^{n-1}\frac{1}{1-T_{[i]}}\cdot x^\alpha.\]
Note that the action of each $T_i$ on any $x^\alpha$ may be realized by suitably multiplying $x^\alpha$ by a monomial in $\Z[[x_1^{\pm 1},x_2^{\pm 1},\ldots,x_n^{\pm 1}]]$. More precisely, we have
\[T_i\cdot x^\alpha=\bigg(x_i^{-d_i}\prod_{\{v_j, v_i\}\in E}x_j\bigg) x^\alpha.\]
This implies
\[T_{[i]}\cdot x^\alpha=\bigg(x_1^{-d_1}x_2^{-d_2}\cdots x_i^{-d_i}\prod_{k=1}^{i}\prod_{\{v_j,v_k\}\in E}x_j\bigg)x^\alpha,\]
so that we obtain
\begin{equation}\label{T_{[i]}x}
\frac{1}{1-T_{[i]}}\cdot x^\alpha=\bigg(\frac{x_1^{d_1}x_2^{d_2}\cdots x_i^{d_i}}{x_1^{d_1}x_2^{d_2}\cdots x_i^{d_i}-\prod_{k=1}^i\prod_{\{v_j,v_k\}\in E}x_j}\bigg)x^\alpha,
\end{equation}
and finally
\begin{equation}\label{taux}\tau\cdot x^\alpha=\bigg(\prod_{i=1}^{n-1}\frac{x_1^{d_1}x_2^{d_2}\cdots x_i^{d_i}}{x_1^{d_1}x_2^{d_2}\cdots x_i^{d_i}-\prod_{k=1}^i\prod_{\{v_j,v_k\}\in E}x_j}\bigg)x^\alpha.
\end{equation}
Identities \eqref{T_{[i]}x} and \eqref{taux} have to be intended in the following way: set $X_i=x_1^{d_1}x_2^{d_2}\cdots x_i^{d_i}$ and $Y_i=\prod_{k=1}^i\prod_{\{v_j,v_k\}\in E}x_j$, then expand the right-hand side in \eqref{taux} as a power series in $Y_i/X_i$, so that
\[T_{[i]}\cdot x^\alpha = \frac{X_i}{X_i-Y_i}x^\alpha=\sum_{n\geq 0}\bigg(\frac{Y_i}{X_i}\bigg)^nx^\alpha.\]
\begin{ex}[The complete graph $G=K_n$]
We have
\[T_i\cdot x^\alpha=x^\alpha\,\frac{x_1x_2\cdots x_n}{x_i^n},\]
so that
\[T_{[i]}\cdot x^\alpha=x^\alpha\,\frac{(x_1x_2\cdots x_n)^i}{(x_1x_2\cdots x_i)^n},\]
and finally
\[\tau(x)=\prod_{i=1}^{n-1}\frac{(x_1x_2\cdots x_i)^n}{(x_1x_2\cdots x_i)^n-(x_1x_2\cdots x_n)^i}.\]
By expanding as a power series in $(x_1x_2\cdots x_n)^i/(x_1x_2\cdots x_i)^n$ we recover
\[\mathcal{H}_\alpha(x)=x^\alpha\,\prod_{i=1}^{n-1}\sum_{k\geq 0}\frac{(x_1x_2\cdots x_n)^{ki}}{(x_1x_2\cdots x_i)^{kn}}.\]
So we have the following theorem.
\begin{thm}
For each graph $\mathcal{G}=(V,E)$ there exists a formal power series $\tau(x)\in\Z[[x_1^{\pm 1},x_2^{\pm 1},\ldots,x_n^{\pm 1}]]$ such that
\[\mathcal{H}_\alpha(x)=x^\alpha\tau(x),\]
for all $\alpha\in\Z^n$.
\end{thm}
\end{ex}
Elements $T_{[i]}$'s not only concur in giving an explicit expression for the operator $\tau$, they also are algebrically independent and generate a subalgebra $\C[G]_\geq$ of $\C[G]$. More precisely, let $\C[G]_\geq =\C[T_{[1]},T_{[2]},\ldots,T_{[n-1]}]$ and observe that, via Theorem \ref{Tlambda}, $\C[G]_\geq$ exactly is the subalgebra generated by all $T^\lambda$'s with $\lambda\in P_n$. A wider set of generators of $\C[G]_{\geq}$ is obtained by setting
\[T_{[i,j]}=T_{[i]}T_{[i+1]}\cdots T_{[j-1]} \text{ for all }1\leq i<j\leq n.\]
Note that each $T_{[i,j]}$ is a Yamanouchi toppling sequence associated with a tableaux whose conjugate shape consists of consecutive integers. Obviously $T_{[i]}=T_{[i,i+1]}$ so that the $T_{[i,j]}$'s generate the whole algebra $\C[G]_{\geq}$. On the other hand, the expression of each $T^\lambda$ in terms of such generators is not unique. In order to find a \textit{reduced decomposition} of $T^{\lambda}=T_{[\lambda_1']}T_{[\lambda_2']}\cdots$ in terms of the $T_{[i,j]}$'s we take the following path. Rearrange and associate the $T_{[\lambda_i']}$'s so that
\[T^\lambda=(T_{[i_1]}\cdots T_{[j_1]})(T_{[i_{2}]} \cdots T_{[j_2}])\ldots(T_{[i_{k}]}\cdots T_{[j_k}]),\]
and each of the sequences $(i_h,\ldots, j_h) $'s consists of  increasing consecutive integers. Then the reduced decomposition of $T^\lambda$ is
\[T^\lambda=T_{[i_1,j_1+1]}T_{[i_{2},j_2+1]}\cdots T_{[i_{k},j_k+1]}.\]
Let us explicit the idea by means of a guiding example.
\begin{ex}
If $\lambda=(8,7,4,3,2,2,1)$ then $\lambda'=(7,6,4,3,2,2,2,1)$. Then we recover
\begin{multline*}T^\lambda=T_{[7]}T_{[6]}T_{[4]}T_{[3]}T_{[2]}T_{[2]}T_{[2]}T_{[1]}=(T_{[1]}T_{[2]}T_{[3]}T_{[4]})(T_{[2]})(T_{[2]})(T_{[6]}T_{[7]})\\=T_{[1,5]}T_{[2,3]}^2T_{[6,8]}.\end{multline*}
A reduced decomposition of $T^\lambda$ is then $T_{[1,5]}T_{[2,3]}^2T_{[6,8]}$.
\end{ex}
In general, any $T^\lambda$ admits several reduced decompositions. For instance, both $T_{[1,3]}T_{[2,4]}$ and $T_{[1,4]}T_{[2,3]}$ are reduced decompositions of $T_1^4T_2^3T_3$. Hereafter, the total number of (reduced and non reduced) decompositions of $T^\lambda$ will be denoted by $C(\lambda)$ or by $C_{\alpha,\beta}$ if $\alpha,\beta\in\Z^n$ and $\beta=T^\lambda(\alpha)$. A decomposition $T^\lambda=T_{[i_1,j_1]}T_{[i_2,j_2]}\cdots T_{[i_l,j_l]}$ is said to be \textit{square free} if and only if each generator occurs with multiplicity at most one. For instance, $T_{[1,2]}T_{[1,3]}T_{[2,5]}$ is square free but $T_{[1,2]}^2T_{[2,3]}T_{[2,5]}$ is not. Also, note that $T_{[1,2]}T_{[1,3]}T_{[2,5]}=T_{[1,2]}^2T_{[2,3]}T_{[2,5]}$ so that any $T^\lambda$ may have both square free and non square free decompositions. Now, consider the operator $\hat{\tau}$ defined by
\begin{equation}
\label{hattau}\hat{\tau}=\prod_{1\leq i<j\leq n}\frac{1}{1-T_{[i,j]}}.
\end{equation}
We recover
\[\hat{\tau}=\sum_{\lambda\in P_n}C(\lambda)T^\lambda,\]
so that we may write
\[\hat{\mathcal{H}}_\alpha(x)=\hat{\tau}\cdot x^\alpha=\sum_{\lambda\in P_n}C(\lambda)\,x^{T^\lambda(\alpha)}=\sum_{\beta\leq\alpha}C_{\alpha,\beta}\,x^\beta.\]
Roughly speaking the series $\hat{\mathcal{H}}_\alpha(x)$, as well as $\mathcal{H}_\alpha(x)$, is a generating series for the set $\mathcal{H}_\alpha$. However, when $\hat{\tau}$ acts on the monomial $x^\alpha$, each $\beta\in\mathcal{H}_\alpha$ is obtained $C_{\alpha,\beta}$ times.

At this point, any element in $G$, and in particular any $T^\lambda$, can be expressed in terms of three families of generators of the algebra $\C[G]_\geq$. Namely, we have the sets $\{T_i\,|\,i=1,2,\ldots,n\}$, $\{T_{[i]}\,|\,i=1,2,\ldots,n\}$ and $\{T_{[i,j]}\,|\,1\leq i<j\leq n\}$. Each family of generators gives rise to a notion of length of the  decomposition, that is the total number of generators involved. More precisely, each $T^\lambda$ admits a unique reduced decomposition in terms of the $T_i$'s, whose length $\ell_1$ equals the size of the partition $\lambda\in P_n$. Analogously, such a $T^\lambda$ can be written in a unique way in terms of the $T_{[i]}$'s. In particular, we have $T^\lambda=T_{[\lambda_1']}T_{[\lambda_2']}\cdots$ and $\ell_2=\lambda_1$ generators are involved. Moreover, each of the $C(\lambda)$ pairwise different decompositions of $T^\lambda$ in terms of the $T_{[i,j]}$'s involves a certain number, say $\ell_3$, of generators. Finally, each reduced decomposition of $T^\lambda$ in terms of the $T_{[i,j]}$'s involves a cerrtain number, say $d$, of pairwise distinct generators. As a matter of fact, both $\ell_1$ and $\ell_2$ can be easily recovered once that a decomposition of $T^\lambda$ in terms of the $T_{[i,j]}$'s is known. Indeed, assume we have
\[T^\lambda=T_{[i_1,j_1]}^{a_1}T_{[i_2,j_2]}^{a_2}\cdots T_{[i_d,j_d]}^{a_d},\]
with the $[i_k,j_k]$'s all distinct. Then, it is not difficult to see that\footnote{where we assume $\binom{1}{2}=0$}
\[\ell_1=\sum_{h=1}^da_h\left(\binom{j_h}{2}-\binom{i_h}{2}\right),\]
and
\[\ell_2=\sum_{h=1}^da_h(j_h-i_h).\]
\begin{ex}
Consider again $\lambda=(8,7,4,3,2,2,1)$ so that
\[T^\lambda=T_1^8T_2^7T_3^4T_4^3T_5^2T_6^2T_7.\]
We have $\ell_1=|\lambda|=8+7+4+3+2+2+1=27$. Moreover, being $\lambda'=(7,6,4,3,2,2,2,1)$ then $T^\lambda=T_{[7]}T_{[6]}T_{[4]}T_{[3]}T_{[2]}T_{[2]}T_{[2]}T_{[1]}$ and in fact $\ell_2=\lambda_1=8$. Now, consider the following decomposition of $T^\lambda$ in terms of the $T_{[i,j]}$'s,
\[T^\lambda=T_{[1,5]}T_{[2,3]}^2T_{[6,8]}.\]
It involves $\ell_3=4$ generators, and $d=3$ among them are distinct. Finally, note that from $T^\lambda=T_{[1,5]}T_{[2,3]}^2T_{[6,8]}$ we recover
\[\ell_1=\binom{5}{2}-\binom{1}{2}+2\left(\binom{3}{2}-\binom{2}{2}\right)+\binom{8}{2}-\binom{6}{2}=27,\]
and also
\[\ell_2=(5-1)+2(3-2)+(8-6)=8.\]
\end{ex}
\begin{thm}
Set
\[\hat{\tau}(z_1,z_2,z_3,q)=\sum_{\lambda\in P_n}\bigg(\sum_{\ell_1,\ell_2,\ell_3,d}z_1^{\ell_1}z_2^{\ell_2} z_3^{\ell_3}q^{d}\bigg)\,T^\lambda,\]
where the values of $\ell_1,\ell_2,\ell_3,d$ range over all pairwise distinct decompositions of $T^\lambda$ in terms of the generators $T_{[i,j]}$'s.
Then we have
\begin{equation}
\label{tqz2}\hat{\tau}(z_1,z_2,z_3,q)=\prod_{1\leq i<j\leq n}\frac{1-(1-q)T_{[i,j]}z_3z_2^{j-i}z_1^{\binom{j}{2}-\binom{i}{2}}}{1-T_{[i,j]}z_3z_2^{j-i}z_1^{\binom{j}{2}-\binom{i}{2}}}.
\end{equation}
\end{thm}
\begin{proof}
We have
\begin{multline*}\hat{\tau}(z_1,z_2,z_3,q)=\prod_{1\leq i<j\leq n}\left(1+\frac{qT_{[i,j]}z_3z_2^{j-i}z_1^{\binom{j}{2}-\binom{i}{2}}}{1-T_{[i,j]}z_3z_2^{j-i}z_1^{\binom{j}{2}-\binom{i}{2}}}\right)\\
=\prod_{1\leq i<j\leq n}\sum_{k\geq 0}T_{[i,j]}^kqz_1^{k\left(\binom{j}{2}-\binom{i}{2}\right)}z_2^{k(j-i)}z_3^k.\end{multline*}
Then, straightforward computations will give \eqref{tqz2}.
\end{proof}
We may define a parametrized version of the series $\hat{\mathcal{H}}_\alpha(x)$ by setting
\begin{equation}\label{hatH}
\hat{\mathcal{H}}_\alpha(z_1,z_2,z_3,q;x)=\hat{\tau}(z_1,z_2,z_3,q)\cdot x^\alpha.
\end{equation}
We recover

\[\hat{\mathcal{H}}_\alpha(z_1,z_2,z_3,q;x)=\sum_{\beta\leq\alpha}C_{\alpha,\beta}(z_1,z_2,z_3,q) x^\beta,\]
where the polynomial $C_{\alpha,\beta}(z_1,z_2,z_3,t)$ stores the values of $\ell_1,\ell_2,\ell_3,d$ relative to all pairwise distinct decompositions of the unique $T^\lambda$ such that $\beta=T^\lambda(\alpha)$,
\[C_{\alpha,\beta}(z_1,z_2,z_3,q)=\sum_{\ell_1,\ell_2,\ell_3,d} z_1^{\ell_1}z_2^{\ell_2}z_3^{\ell_3}q^{d}.\]
In particular,  we remark that for all $\beta\in\mathcal{H}_\alpha$ we have
\[C_{\alpha,\beta}(1,1,1,1)=C_{\alpha,\beta}.\]
Moreover, we note that from \eqref{hatH} we gain a combinatorial description of the series $\hat{\mathcal{K}}_\alpha(z_1,z_2,z_3,q;x)$ defined by
\[\hat{\mathcal{K}}_\alpha(z_1,z_2,z_3,q;x)=\hat{\tau}(z_1,z_2,z_3,q)^{-1}\cdot x^{\alpha}.\]
In fact, being $\hat{\tau}(z_1,z_2,z_3,q)^{-1}=\hat{\tau}\left(z_1,z_2,(1-q)z_3,\frac{q}{q-1}\right)$, we gain
\[\hat{\mathcal{K}}_\alpha(z_1,z_2,z_3,q;x)=\hat{\mathcal{H}}_\alpha\left(z_1,z_2,(1-q)z_3,\frac{q}{q-1};x\right)\]
so that the following combinatorial description is obtained,
\[\hat{\mathcal{K}}_\alpha(z_1,z_2,z_3,q;x)=\sum_{\beta\leq \alpha}
\left(\sum_{\ell_1,\ell_2,\ell_3,d} z_1^{\ell_1}z_2^{\ell_2}z_3^{\ell_3}(1-q)^{\ell_3-d}(-q)^d\right) x^\beta.\]
Since $\ell_1\geq \ell_2\geq\ell_3\geq d\geq 0$ then also the coefficient $C'_{\alpha,\beta}(z_1,z_2,z_3,q)$ of $x^\beta$ in $\hat{\mathcal{K}}_\alpha(z_1,z_2,z_3,q;x)$ is a polynomial with integer coefficients. It is related to $C_{\alpha,\beta}(z_1,z_2,z_3,q)$ by means of
\[C'_{\alpha,\beta}(z_1,z_2,z_3,q)=C_{\alpha,\beta}\left(z_1,z_2,(1-q)z_3,\frac{q}{q-1}\right).\]
By setting $z_1=z_2=z_3=q=1$ in $\hat{\mathcal{K}}_\alpha(z_1,z_2,z_3,q;x)$ the only decompositions giving a nonzero contribution are those for which $\ell_3-d=0$, that is exactly square free decompositions.
\section{Hall-Littlewood symmetric functions and Yamanouchi toppling}
Let $x=\{x_1,x_2,\ldots,x_n\}$ and denote by $\Lambda(x)$ the ring of symmetric polynomials in $x$ with integer coefficients. For each positive integer $i$, let $h_i(x)$ denote the $i$th complete homogeneous symmetric polynomial, so that we have $\Lambda(x)=\mathbb{Z}[h_1,h_2,\ldots]$ and then any $f(x)\in\Lambda(x)$ may be written in a unique way as a polynomial in the $h_i(x)$'s with integer coefficients. In particular, for all $\alpha\in\N^n$ we define
\[s_\alpha(x)=\det(h_{\alpha_i+j-i}(x))_{1\leq i,j\leq n}.\]
The Jacobi-Trudi formula assures us that $s_\alpha(x)$ is a Schur polynomial if $\alpha_1\geq \alpha_2\geq\ldots\geq\alpha_n\geq 0$, that is if $\alpha$ is an integer partition. Moreover, by swapping $i$th row and $(i+1)$th row in the determinant above we have
\begin{equation}\label{s_inv}
s_{\beta}(x)=-s_{\alpha}(x) \text{ with }\beta=(\alpha_1,\ldots,\alpha_{i+1}-1,\alpha_i+1,\ldots,\alpha_n).
\end{equation}
This is to say that any $s_\alpha(x)$ is zero or there is a partition $\lambda$ such that $s_\alpha(x)=\pm s_\lambda(x)$. Now, consider the linear functional
\[E\colon x^\alpha\mapsto \begin{cases}s_\alpha(x),&\text{ if }\alpha\in\N^n;\\0,&\text{ if }\alpha\notin\N^n.\end{cases}\]
Hence, one may define symmetric polynomials by means of
\[E\,\hat{\mathcal{H}}_\alpha(z_1,z_2,z_3,q;x)\text{ and }E\,\hat{\mathcal{K}}_\alpha(z_1,z_2,z_3,q;x).\]
Next theorem states a first obvious but important fact.
\begin{thm}
For any graph $\mathcal{G}$, both
\[\{E\,\hat{\mathcal{H}}_\alpha(z_1,z_2,z_3,q;x)\}_{\alpha_1\geq \alpha_2\geq \ldots\geq \alpha_n\geq 0} \text{ and }\{E\,\hat{\mathcal{K}}_\alpha(z_1,z_2,z_3,q;x)\}_{\alpha_1\geq \alpha_2\geq \ldots\geq \alpha_n\geq 0},\]
are bases of the ring $\Lambda(z_1,z_2,z_3,q;x)$ of symmetric polynomials in $x$ with coefficients in $\mathbb{Z}[q,z_1,z_2,z_3]$.
\end{thm}
The bases above gain particular interest in view of the special case when Yamanouchi toppling is performed on the simple path $\mathcal{G}=\mathcal{L}$ with edges
\[\{v_1,v_2\},\{v_2,v_3\},\ldots,\{v_{n-1},v_n\}.\]
\begin{lem}
For the graph $\mathcal{G}=\mathcal{L}$ the generator $T_{[i,j]}$ realizes the lowering operator, more precisely for all $1\leq i<j\leq n$ we have
\[T_{[i,j]}(\alpha) =\alpha-\epsilon_i+\epsilon_{j}. \]
\end{lem}

\begin{proof}
For this graph we have
\[T_i(\alpha)=\begin{cases}
\alpha-\epsilon_1+\epsilon_2,& \text{ if }i=1;\\
\alpha-2\epsilon_i+\epsilon_{i-1}+\epsilon_{i+1},& \text{ if }2\leq i\leq n-1;\\
\alpha-\epsilon_n+\epsilon_{n-1},& \text{ if }i=n.\end{cases}\]
Then, it is not difficult to see that
\[T_{[i]}(\alpha)=T_1T_2\cdots T_i(\alpha)=\alpha-\epsilon_i+\epsilon_{i+1} \text{ for all }1\leq i\leq n-1,\]
so that
\[T_{[i,j]}(\alpha)=T_{[i]}T_{[i+1]}\cdots T_{[j-1]}(\alpha)=\alpha-\epsilon_i+\epsilon_{j} \text{ for all }1\leq i<j\leq n.\]
\end{proof}
Let us recall that, if $\alpha$ is an integer partition then the Hall-Littlewood symmetric polynomial $R_\alpha(x;t)$ is defined by
\[R_{\alpha}(x;t)=\sum_{w\in\mathfrak{S}_n}w\bigg(x^\alpha\prod_{1\leq i<j\leq n}\frac{x_i-tx_j}{x_i-x_j}\bigg).\]
\begin{thm}
Let $\mathcal{G}=\mathcal{L}$, then we have
\[R_\alpha(x;t)=E\,\hat{\mathcal{K}}_\alpha(1,1,t,1;x)=\lim_{q\to 1}E\,\hat{\mathcal{H}}_\alpha\bigg(1,1,(1-q)t,\frac{q}{q-1};x\bigg).\]
\end{thm}
\begin{proof}
For all $1\leq i<j\leq n$ let $R_{j,i}\colon\Z^n\to\Z^n$ denote the lowering operator defined by $R_{j,i}(\alpha)=\alpha-\epsilon_i+\epsilon_j$. From \cite{M95} we recover
\begin{equation}\label{lowR}
R_\alpha(x;t)=\bigg\{\prod_{1\leq i<j\leq n}(1-tR_{j,i})\bigg\}\cdot s_\alpha(x),
\end{equation}
where $R\cdot s_\alpha(x)=s_{R(\alpha)}$ for any product $R$ of lowering operators, and $s_{R(\alpha)}=0$ if $R(\alpha)\notin\N^n$. Observe that \eqref{lowR} can rewritten  as
\begin{equation}\label{lowRbis}
R_\alpha(x;t)=E\,\bigg\{\prod_{1\leq i<j\leq n}(1-tR_{j,i})\bigg\}\cdot x^\alpha,
\end{equation}
where we set $R\cdot x^\alpha=x^{R(\alpha)}$ for any product $R$ of lowering operators. Now, set $\mathcal{G}=\mathcal{L}$ so that we have by the above Lemma $T_{[i,j]}=R_{j,i}$ for all $1\leq i<j\leq n$. Then, \eqref{lowRbis} implies
\begin{multline*} R_\alpha(x;t)=E\,\bigg\{\prod_{1\leq i<j\leq n}(1-tT_{[i,j]})\bigg\}\cdot x^\alpha\\
=E\,\hat{\mathcal{K}}_\alpha(1,1,t,1;x)=\lim_{q\to 1}E\,\hat{\mathcal{H}}_\alpha\bigg(1,1,(1-q)t,\frac{q}{q-1};x\bigg).
\end{multline*}
\end{proof}
\begin{rem}[Toppling dominance for $\mathcal{L}$]
A further consequence of the fact that $T_{[i,j]}$ equals the lowering operator $R_{j,i}$ is that the restriction of toppling dominance to the set of all integer partitions of at most $n$ parts reduces, when $\mathcal{G}=\mathcal{L}$, to the classical dominance ordering.
\end{rem}
The coefficient of $s_\beta(x)$ in $R_\alpha(x;t)$ admits a description in terms of square free decompositions of elements in the toppling group of $\mathcal{L}$. More precisely, we recover
\[R_\alpha(x;t)=\sum_{\beta\leq \alpha}C'_{\alpha,\beta}(1,1,t,1) s_{\beta}(x).\]
Now, if $\lambda\in P_n$ is such that $T^\lambda(\alpha)=\beta$ then we can write
\[C'_{\alpha,\beta}(1,1,t,1)=\sum_{\ell_3} (-t)^{\ell_3},\]
where $\ell_3$ ranges over the lengths of all square free decompositions of $T^\lambda$. This gives us the following formula for Hall-Littlewood symmetric polynomials.
\begin{thm}\label{HLcomb}
Let $\alpha$ be an integer partition with at most $n$ parts. Then, we have
\[R_\alpha(x;t)=\sum_{\beta\leq \alpha}\left(\sum_{\ell_3} (-t)^{\ell_3}\right) s_{\beta}(x),\]
where $\ell_3$ ranges over all lengths of all square free decompositions of the unique $T^\lambda$ such that $T^\lambda(\alpha)=\beta$.
\end{thm}
The most interesting transition matrix involving Hall-Littlewood symmetric polynomials arises by expanding Schur polynomials in terms of a normalized version of the $R_\alpha(x;t)$'s which is usually denoted $P_\alpha(x;t)$ \cite{M95}. The entries $K_{\lambda,\alpha}(t)$'s of this matrix, often called $t$-Kostka polynomials, are polynomials in $t$ with positive integer coefficients. A celebrated combinatorial description, due to Lascoux and Sch\"utzenberger \cite{LS}, expresses $K_{\lambda,\alpha}(t)$ as an enumeration of semistandard Young tableaux of shape $\lambda$ and weight $\alpha$ with respect to the \textit{charge} statistic. More recently, Haglund, Haiman, Loher and others have developed a beautiful combinatorial theory for Macdonald polynomials \cite{Hag,Hai,HHL05,HHL08}. This framework provides a new explanation of Lascoux-Schutzenberger's result and extend the combinatorial description from Hall-Littlewood symmetric polynomials to Macdonald polynomials and their non-symmetric generalizations. The problem of finding a satisfactory combinatorial description of $t$-Kostka polynomials, as well as that of finding an interpretation of Macdonald polynomials in terms of the toppling game, still remains open. When $t=1$, $K_{\lambda,\alpha}(t)$ reduces to the Kostka number $K_{\lambda,\alpha}$. We close this section by giving an expression of Kostka numbers in terms of the coefficients $C_{\alpha,\beta}$'s.
\begin{thm}
Let $\mathcal{G}=\mathcal{L}$, let $\lambda$ and $\mu$ be integer partitions with at most $n$ parts, and assume $\lambda$ dominates $\mu$. Then, we have
\[K_{\lambda,\mu}=\sum_{w}(-1)^w C_{\mu,w(\lambda)+w(\delta)-\delta},\]
with $w$ ranging over the symmetric group on $1,2,\ldots,n$ and with $\delta=(n-1,n-2,\ldots,0)$.
\end{thm}
\begin{proof}
The classical definition of the Schur polynomial states that
\[s_\lambda(x)=\frac{\det(x_i^{\lambda_j+n-j})_{1\leq i,j\leq n}}{\det(x_i^{n-j})_{1\leq i,j\leq n}}=\det(x_i^{\lambda_j+i-j})_{1\leq i,j\leq n}\prod_{i<j}\bigg(1-\frac{x_j}{x_i}\bigg)^{-1}.\]
By expanding each factor $(1-x_j/x_i)^{-1}$ as a formal power series in $x_j/x_i$ the identity above still is true and, in particular, we obtain
\begin{align*}s_\lambda(x)&=\prod_{1\leq i<j\leq n}\frac{1}{1-\frac{x_j}{x_i}}\,\det(x_i^{\lambda_j+n-j})_{1\leq i,j\leq n}\\
&= \hat{\tau}\cdot\bigg(\sum_{w\in\mathfrak{S}_n}(-1)^w x^{w(\lambda)+w(\delta)-\delta}\bigg)\\
&=\sum_{w}\sum_{\beta\leq w(\lambda)+w(\delta)-\delta} (-1)^w\,C_{\beta,w(\lambda)+w(\delta)-\delta}\,x^\beta,
\end{align*}
where $w$ ranges over all permutations of $1,2,\ldots,n$ and $\delta=(n-1,n-2,\ldots,0)$. On the other hand, it is well known that
\[s_\lambda(x)=\sum_{\mu}K_{\lambda,\mu}m_\mu(x).\]
By comparing the coefficient of $x^\mu$ we recover
\[K_{\lambda,\mu}=\sum_{w}(-1)^w C_{\mu,w(\lambda)+w(\delta)-\delta}.\]
\end{proof}
\section{Classical orthogonal polynomials and Yamanouchi toppling}
Let us recall the notion of orthogonal polynomial system \cite{C78}. Assume that a linear functional $L\colon\mathbb{R}[t]\to\mathbb{R}$ is given. An \textit{orthogonal polynomial system} associated with $L$ is a polynomial sequence $\{p_n(t)\}_{n\in\N}$ such that $p_n(t)\in\mathbb{R}[t]$ and  $deg\,p_n=n$ for all $n\in\N$, and such that
\[L\,p_n(t)p_m(t)=0 \text{ if and only if }n\neq m.\]
Let $n$ be a positive integer and let $\mathcal{L}_n$ denote the simple path with $n$ vertices and with edges
\[\{v_1,v_2\},\{v_2,v_3\},\ldots,\{v_{n-1},v_n\}.\]
Denote by $\hat{\tau}_n$ the operator $\hat{\tau}$ relative to $\mathcal{L}_n$, that is
\[\hat{\tau}_n=\prod_{1\leq i<j\leq n}\frac{1}{1-T_{[i,j]}}.\]
Moreover, set $\boldsymbol{x}_n=\{x_1,x_2,\ldots,x_n\}$ and, for all $\alpha\in\Z^n$, define $q_\alpha(\boldsymbol{x}_n)$ to be the unique polynomial such that
\begin{equation}\label{q} \hat{\tau}_n\cdot q_\alpha(\bs{x}_n)=\bs{x}_n^\alpha,\end{equation}
where we set $\bs{x}_n^\alpha=x_1^{\alpha_1}x_2^{\alpha_2}\cdots x_n^{\alpha_n}$. Recall that we have
\[\hat{\tau}_n^{-1}=\prod_{1\leq i<j\leq n}(1-T_{[i,j]})=\sum_{\lambda\in P_n}\bigg(\sum_{\ell_3}(-1)^{\ell_3}\bigg)T^\lambda,\]
where $\ell_3$ ranges over all lengths of all square free decompositions of the fixed $T^\lambda$. Thus, we recover the following combinatorial formula for $q_\alpha(\bs{x}_n)$,
\[q_\alpha(\boldsymbol{x}_n)=\sum_{\beta\leq \alpha}\bigg(\sum_{\ell_3}(-1)^{\ell_3}\bigg)x^\beta.\]
Since the size of a configuration is preserved by any toppling sequence, then $q_\alpha(\boldsymbol{x}_n)$ is a homogeneous polynomial in $x_1,x_2,\ldots,x_n$ of total degree $|\alpha|$. To show how the polynomials $q_\alpha(\boldsymbol{x}_n)$'s are related to orthogonal polynomials systems we need to manipulate polynomials with an arbitrary large number of variables at the same time. To this aim, we set $\boldsymbol{x}=\{x_1,x_2,\ldots\}$ and define
\[\mathbb{R}[\boldsymbol{x}]=\bigcup_{n\geq 1}\mathbb{R}[\boldsymbol{x}_n].\]
Moreover, we will use maps $E\colon \mathbb{R}[\boldsymbol{x}]\to\mathbb{R}$ such that
\begin{enumerate}
\item for all $n\geq 1$ the restriction $E\colon\mathbb{R}[\boldsymbol{x}_n]\to\mathbb{R}$ is a linear functional;
\item for all $n\geq 2$, for all $p\in\mathbb{R}[\boldsymbol{x}_n]$ and for all $w\in\mathfrak{S}_n$,
\[E\,p(x_{w(1)},x_{w(2)},\ldots,x_{w(n)})=E\,p(x_1,x_2,\ldots,x_n).\]
\end{enumerate}
We will name $E$ \textit{symmetric functional}. Once that a symmetric functional $E$ is given, for all $i\geq 1$ we may define a \textit{conditional operator} $E_i\colon \mathbb{R}[\boldsymbol{x}]\to\mathbb{R}[x_i]$. Such an operator is uniquely determined by
\[E_i\,\boldsymbol{x}_n^\alpha=x_i^{\alpha_i} E \boldsymbol{x}_n^\alpha x_i^{-\alpha_i} \text{ for all }n\in\N,\text{ and for all }\alpha\in\N^n.\]
Roughly speaking, $E_i$ acts on $\mathbb{R}[x_1,\ldots, x_{i-1},x_{i+1},\ldots]$ as $E$ acts, and fixes each polynomial in $\mathbb{R}[x_i]$. We will say that the variables $x_1,x_2,\ldots$ are \textit{independent} with respect to the functional $E$ if and only if $E=E\,E_i$ for all $i\geq 1$, that is if and only if
\[EE_i\,p=E\,p, \text{ for all }p\in\mathbb{R}[\boldsymbol{x}]\text{ and for all }i\geq 1.\]
Note that the degree of $E_i\,q_\alpha(\bs{x}_n)$ does not exceed the maximum $k\in\N$ such that $x_i^k$ occurs in $q_\alpha(x)$. Hence, we define $\{p_n(t)\}_{n\in\mathbb{N}}$ to be the unique polynomial sequence such that $p_0(t)=1$, and such that
\begin{equation}\label{p_n}
p_{n}(x_{n+1})=E_{n+1}\,q_{(n,n,\ldots,n,0)}(\bs{x}_{n+1}),\text{ for all }n\geq 1.
\end{equation}
The following Theorem states an orthogonality relation for the polynomials  defined in \eqref{p_n}.
\begin{thm}\label{thm:orth}
Let $E\colon\mathbb{R}[\bs{x}]\to\mathbb{R}$ be a symmetric functional, let $x_1,x_2,\ldots$ be independent with respect to $E$, and let $\{p_n(t)\}_{n\in\N}$ denote the unique polynomial sequence such that $p_0(t)=1$ and satisfying \eqref{p_n}. Then, for all $x_i\in\bs{x}$ we have
\[E\,p_n(x_{i})p_m(x_{i})=0 \text{ for all }n,m\in\N \text{ such that }n\neq m.\]
Moreover, if $\deg\,p_n=n$ for all $n\in\N$ then we also have
\[E\,p_n(x_{i})p_n(x_{i})\neq 0 \text{ for all }n\in\N.\]
\end{thm}
\begin{proof}
Let $n,m\in\N$ with $0\leq m < n$. Since $E=E\,E_{n+1}$ the, by comparing \eqref{q} and \eqref{p_n}, we obtain
\[E\,x_{n+1}^m\,p_{n}(x_{n+1})=E\,x_{n+1}^mx_2x_3^2\cdots x_{n}^{n-1}\prod_{1\leq i<j\leq n+1}(x_i-x_j).\]
Let $w=(n+1,m+1)\in\mathfrak{S}_{n+1}$ denote the transposition exchanging $n+1$ and $m+1$. Since $E$ is a symmetric functional then we have
\begin{multline*} E\,x_{n+1}^mx_2x_3^2\cdots x_n^{n-1}\prod_{1\leq i<j\leq n+1}(x_i-x_j)\\
=E\,x_{w(n+1)}^mx_{w(2)}x_{w(3)}^2\cdots x_{w(n)}^{n-1}\prod_{1\leq i<j\leq n+1}(x_{w(i)}-x_{w(j)}).
\end{multline*}
On the other hand, it is easily seen that
\begin{multline*} x_{w(n+1)}^mx_{w(2)}x_{w(3)}^2\cdots x_{w(n)}^{n-1}\prod_{1\leq i<j\leq n+1}(x_{w(i)}-x_{w(j)})\\
=-x_{n+1}^mx_2x_3^2\cdots x_n^{n-1}\prod_{1\leq i<j\leq n+1}(x_i-x_j).
\end{multline*}
This forces $E\,x_{n+1}^mp_{n}(x_{n+1})=-E\,x_{n+1}^mp_{n}(x_{n+1})=0$ for all $0\leq m< n$. By linearity we recover,
\[E\,p_m(x_{n+1})p_n(x_{n+1})=0 \text{ whenever }0\leq m< n.\]
Of course, the case $0\leq n<m$ is analogous so that we conclude
\[E\,p_m(x_{n+1})p_n(x_{n+1})=0 \text{ for all }n,m\in\N \text{ such that }n\neq m.\]
Besides, the symmetry of $E$ assures us that we may replace $x_{n+1}$ with any of the variables $x_i$'s.

Let $p_{n,n}$ denote the leading coefficient of $p_n(x_{n+1})$, so that we may write
\small\[E\,p_n(x_{n+1})p_n(x_{n+1})=p_{n,n} E\,x_{n+1}^n\,p_n(x_{n+1})=p_{n,n}\,EE_{n+1}\,x_{n+1}^nq_{n,n,\ldots,n,0}(\bs{x}_{n+1}).\]\normalsize
From $E=EE_{n+1}$ we obtain
\[E\,p_n(x_{n+1})p_n(x_{n+1})=p_{n,n}E\,x_2x_3^2\cdots x_{n}^{n-1}x_{n+1}^{n}\prod_{1\leq i<j\leq n+1}(x_i-x_j),\]
and finally
\[E\,p_n(x_{n+1})p_n(x_{n+1})=p_{n,n}E\,q_{(n,n,\ldots,n,n)}(\bs{x}_{n+1}).\]
Therefore, by comparing \eqref{q} and \eqref{p_n} it is easy to see that
\[p_{n,n}=E\,x_1^{n-1}x_2^{n-1}\cdots x_n^{n-1}\prod_{1\leq i<j\leq n}(x_i-x_j)=E\,q_{(n-1,n-1,\ldots,n-1,n-1)}(\bs{x}_{n}).\]
We obtain
\[E\,p_n(x_{n+1})p_n(x_{n+1})=p_{n,n}p_{n+1,n+1}.\]
so hence
\[E\,p_n(x_{n+1})p_n(x_{n+1})\neq 0 \text{ for all }n\in\N,\]
whenever $p_{n,n}\neq 0$ for all $n\geq 1$.
\end{proof}
Theorem \ref{thm:orth} gives us an explicit way to build up an orthogonal polynomial system associated with any linear functional $L\colon\mathbb{R}[t]\to\mathbb{R}$, provided it exists. In fact, define $E\colon \mathbb{R}[\bs{x}]\to\mathbb{R}$ to be the unique symmetric functional such that
\[L\,t^k=E\,x_i^k \text{ for all }i,k\in\N, i\neq 0.\]
One can easily check that $E=EE_i$ for all $i\in\N$ with $i\neq 0$. Thus we may define the polynomial sequence $\{p_n(t)\}_{n\in\N}$ such that $p_0(t)=1$ and satisfying \eqref{p_n}. Theorem above assures us that, if $\deg\,p_n=n$ for all $n\in\N$, then we have
\[L\,p_n(t)p_n(t)=E\,p_n(x_{n+1})p_n(x_{n+1})=0 \text{ if and only if }n\neq m,\]
and thus $\{p_n(t)\}_{n\in\N}$ is an orthogonal polynomial system associated with $L$. In turn, this means that the following combinatorial
description of orthogonal polynomial systems can be given.
\begin{thm}[A combinatorial formula for orthogonal polynomials]
Assume that $\{p_n(t)\}_{n\in\N}$ is an orthogonal polynomial system with respect to some linear functional $L$, then we have
\[p_{n}(t)=\sum_{\beta\leq (n,n,\ldots,n,0)}\bigg(\sum_{\ell_3} (-1)^{\ell_3}\bigg)a_{\beta_1}a_{\beta_2}\cdots a_{\beta_{n}}t^{\beta_{n+1}},\]
where $a_i=L\,t^i$ denotes the $i$th oment of $L$, and where $\ell_3$ ranges over all lengths of all square free decompositions of the unique $\lambda\in P_{n+1}$ such that $\beta=T^\lambda(n,n,\ldots,n,0)$.
\end{thm}
By comparing this combinatorial formula with Theorem \ref{HLcomb} one realizes the strong analogy between the expansion of an orthogonal polynomial in terms of the moments of the associated linear functional, and the expansion of a Hall-Littlewood symmetric polynomials in terms of the Schur functions. One might go a bit more into this analogy by considering more general families of graphs $\{\mathcal{G}_n\}_{n\geq 1}$, with $\mathcal{G}_n$ having $n$ vertices and $\mathcal{G}_{n+1}$ obtained from $\mathcal{G}_n$ by adding a vertex $v_{n+1}$ and a certain number of edges. Hence, analogues of equations \eqref{q} and \eqref{p_n} can be given for a general $\mathcal{G}_{n+1}$, with $\hat{\tau}$ possibly replaced by $\hat{\tau}(z_1,z_2,z_3,q)$. Thus, a polynomial sequence $\{p_n(z_1,z_2,z_3,q;t)\}_{n\in\N}$ associated with any family $\{\mathcal{G}_n\}_{n\geq 1}$ is obtained. It reduces to classical orthogonal polynomial systems when $\mathcal{G}_n=\mathcal{L}_n$ and when $z_1=z_2=z_3=q=1$. This opens the way toward a general combinatorial theory for the analogues of the classical orthogonal polynomials, as well as of the classical symmetric functions, defined starting from a general family of graphs $\{\mathcal{G}_n\}_{n\geq 1}$.
\begin{rem}
The coding of orthogonal polynomials via symmetric functionals is at root of a deep connection among orthogonal polynomial systems and the invariant theory of binary forms. More on this subject, including a treatment of multivariable orthogonal polynomials, can be found in \cite{PSS14}.
\end{rem}
\end{document}